\documentclass[12pt]{article}

\usepackage{amsmath,amsfonts,amssymb,latexsym,amsthm,verbatim,a4wide}

 \theoremstyle{plain}
  \newtheorem{theorem}{Theorem}[section]
  
  \newtheorem{proposition}[theorem]{Proposition}
  \newtheorem{lemma}[theorem]{Lemma}  
  \newtheorem{definition}[theorem]{Definition}
  
  \newtheorem{remark}[theorem]{Remark}
 \newtheorem{example}[theorem]{Example}
\newtheorem{condition}{Condition}

\newcommand{\ep}{{\mathbb {E}}}

\newcommand{\Z}{{\mathbb{Z}}}
\newcommand{\re}{{\mathbb{R}}}
\newcommand{\vc}[1]{{\mathbf #1}}
\newcommand{\blah}[1]{}
\newcommand{\Ent}{{\rm Ent}}
\newcommand{\var}{{\rm Var}}

\newcommand{\diy}{\begin{displaystyle}}
\newcommand{\eiy}{\end{displaystyle}}
\newcommand{\GamV}[1]{\Gamma^{(V)}_{#1}}
\newcommand{\Gam}[1]{\Gamma_{#1}}
\newcommand{\EEval}[1]{{\mathcal{E}}^{(#1)}}
\newcommand{\EE}{\EEval{V}}
\newcommand{\EEt}{\EEval{V_t}}
\newcommand{\EEd}[2]{\EEval{V}_{#1#2}}
\newcommand{\EEsym}[2]{E^{(V,c)}_{#1#2}}
\newcommand{\ZZ}{{\mathcal{Z}}}
\newcommand{\psh}{p^{\#}}

\newcommand{\II}{{\mathbb{I}}}

\title{A discrete log-Sobolev inequality  under a Bakry-\'{E}mery type condition}
\author{Oliver Johnson\thanks{School of Mathematics, University of Bristol, University Walk, Bristol, BS8 1TW, UK. 
Email {\tt maotj@bristol.ac.uk} }}
\date{\today}

\begin{document}

\maketitle 

\begin{abstract} \noindent
We consider probability mass functions $V$ supported on the positive integers
using arguments introduced by Caputo, Dai Pra and Posta, based on  a  Bakry--\'{E}mery condition for a
 Markov birth and death operator with invariant measure $V$.
 Under this condition, we prove a new
modified logarithmic Sobolev inequality, generalizing and strengthening results of Wu,  Bobkov and Ledoux,
and Caputo, Dai Pra and Posta. We show how this inequality implies   results including concentration of measure and hypercontractivity, and discuss how it may
extend to higher dimensions.
\end{abstract}

\section{Introduction and main results}

In their classic 1985 paper \cite{bakry}, Bakry and \'{E}mery introduced the $\Gamma_2$ operator and showed
that  (assuming the  Bakry--\'{E}mery condition)  it
could  be used to prove  results such as Poincar\'{e} and log-Sobolev inequalities.  The books by Bakry, Gentil and Ledoux \cite{bakry2} and by
Guionnet and Zegarlinski \cite{guionnet}  review this theory.  We give a brief overview in
Section \ref{sec:contbakryemery}.

Historically, this analysis 
was  restricted to  functions on continuous spaces such as $\re^d$, or more generally
Riemannian manifolds. However,
more recently it was extended to discrete spaces by Caputo,
Dai Pra and Posta \cite{caputo},  by considering the evolution of a birth and death Markov chain. This analysis built on the work of Chen and others
(see for example \cite{chen8}), who used properties of birth and death chains to bound the spectral gap.
We use a version of these methods here;
we fix  probability mass function $V$ whose support is the whole of  $\Z_+$ and fix
the generator $L_V$ of a birth and death Markov chain with invariant measure $V$:
\begin{definition} \label{def:lv} Given a fixed probability mass function $V$,
 write $L_V$ for the operator acting on functions, and $L_V^*$ for the adjoint with respect to counting measure by:
\begin{eqnarray}
L_V f(x) & := & \left( f(x+1) - f(x) \right) - \frac{V(x-1)}{V(x)} \left( f(x) - f(x-1) \right),  \label{eq:lvdef} \\
L_V^* f(x) & := &  f(x-1) - \left(1 + \frac{V(x-1)}{V(x)} \right) f(x) + \frac{ V(x)}{V(x+1)} f(x+1).  \label{eq:lvadjdef}
\end{eqnarray}
In Equation \eqref{eq:lvdef} and throughout, we use the convention that $V(-1) = 0$.
\end{definition}
 In Section \ref{sec:discretemc}
we formally define the resulting operators
$\GamV{1}$ and $\GamV{2}$ and discuss the resulting
 Bakry--\'{E}mery condition
(we refer to this as the inegrated BE($c$) condition, Condition \ref{cond:dbec}). 
In contrast to the continuous case we require average control, rather than pointwise control of the relevant
operators.  However (see Remark \ref{rem:prodrule}), this $\GamV{1}$ operator does not satisfy a product rule, so 
classical
proofs of log-Sobolev inequalities do not  carry over. 

The continuous Bakry-\'{E}mery condition relates to log-concavity of the underlying reference measure,
with the Gaussian playing a distinguished role (see Example \ref{ex:U2bound}).
We use a similar condition here, which corresponds to Assumption A of \cite{caputo} (in the case where the upward
jump rates of the birth and death chain are all equal to 1):
\begin{definition} Given a probability mass function $V$, write 
\begin{equation} \label{eq:EEdef} \EE(x) := \frac{ V(x)^2 - V(x-1) V(x+1)}{V(x) V(x+1)}
= \frac{V(x)}{V(x+1)} - \frac{V(x-1)}{V(x)}. \end{equation}
\end{definition}
\begin{condition}[$c$-log-concavity] \label{cond:clc}
 If $\EE(x) \geq c$ for all $x \in Z_+$, we say that $V$ is $c$-log-concave.
\end{condition}
In Proposition \ref{prop:clogcgivesdbec} we show the integrated BE($c$) condition is implied by $c$-log-concavity.
\cite{caputo} showed that  $c$-log-concavity follows  from
 the ultra log-concavity (ULC) property of Liggett and Pemantle 
\cite{liggett,pemantle}. 
 Hence integrated 
BE($c$) holds for parametric families of random variables including Poisson, binomial and Poisson-binomial (Bernoulli sums).
For the Poisson mass function with mean $\lambda$ (we write $V = \Pi_{\lambda}$), the $\EE(x) \equiv 1/\lambda$,
so $c$-log-concavity holds with $c = 1/\lambda$, which Lemma \ref{lem:cbound} shows is an extreme value.  This 
helps us to understand $\EE(x)$ as a discrete curvature term, in the sense dicussed by Chafa\"{i} in \cite{chafai,chafai5}.
For example,
in \cite[Section 1.3]{chafai} it is remarked that the $M/M/\infty$ queue (corresponding to $V=\Pi_{\lambda}$) can be understood
to have constant curvature.
 
As in \cite{bakry} and \cite{caputo},  in   Section \ref{sec:prooflsi} we prove
 a new (modified) log-Sobolev inequality, Theorem \ref{thm:lsi},
which is the main result of this paper.
In the continuous case, the log-Sobolev inequality holds under the Bakry-\'{E}mery condition (see for example Theorem \ref{thm:contlsi}), whereas our result requires the
 (slightly stronger) $c$-log-concavity condition.

To fix notation,  we
write $\Delta f(x) = f(x+1) -f(x)$ for the right difference operator, and $L f(x) = f(x+1) - 2 f(x) + f(x-1)$.
Given a probability mass function $V$ and function $f$, we write $\var_V(f) = \sum_{x=0}^\infty V(x) (f(x) - \mu_{V,f})^2$, where $\mu_{V,f} = \sum_{x=0}^\infty V(x) f(x)$.
Similarly, we write $\Ent_V(f) = \sum_{x=0}^\infty V(x) f(x) \log f(x) - \mu_{V,f} \log \mu_{V,f}$.

We now state the main result of this paper. 
As discussed in more detail in Remark \ref{rem:compare}
 below, the form of this inequality is suggested by the fact that it holds for the case where $V$ is
Poisson, as proved in \cite[Theorem 1.1]{wu} and \cite[Corollary 2.4]{chafai}.

\begin{theorem}[New modified log-Sobolev inequality] \label{thm:lsi}
Fix probability mass function $V$, whose support is the whole of the positive integers $\Z_+$ and which
satisfies the $c$-log-concavity condition  (Condition \ref{cond:clc}). For any function $f$ with
positive values:
\begin{equation} \label{eq:lsi} \Ent_V(f) \leq \frac{1}{c} \sum_{x=0}^\infty
 V(x) f(x+1)  \left( \log \left( \frac{ f(x+1)}{f(x)} \right) - 1
+ \frac{f(x)}{f(x+1)} \right) . \end{equation}
By the well-known bound $\log 1/u - 1 + u \geq 0$ for all $u > 0$,  the RHS of \eqref{eq:lsi} is positive.
\end{theorem}
Previous work on different forms of log-Sobolev inequalities in discrete settings is discussed and summarised by Bobkov
and Tetali \cite{bobkov11}. In particular, \cite[Proposition 3.6]{bobkov11} gives a hierarchy of different constants and discusses
the implications between them. One particular form of interest is the more standard modified log-Sobolev inequality:
\begin{equation} \label{eq:modlsi} \Ent_V(f) \leq C \sum_{x=0}^\infty V(x) (f(x+1) - f(x)) \left( \log f(x+1) - \log f(x) \right),
\end{equation}
also considered in \cite{boudou} and discussed in Remark \ref{rem:compare}.\ref{it:modlsi} below.
\begin{remark} \label{rem:compare} We discuss Theorem \ref{thm:lsi} in detail, to put it in context:
\begin{enumerate} 
\item Suppose $f(x) = p(x)/V(x)$, for probability mass function $p$.
Using  normalizing constant $K = \left( \sum_{x=0}^\infty p(x+1) V(x)/V(x+1) \right)^{-1}$, then $\psh(x) = K p(x+1) V(x)/V(x+1)$ is a 
probability mass function acting as a weighted version of $p$. Then
\eqref{eq:lsi} means relative entropy $D( p \| q) := \sum_{x=0}^\infty p(x) \log( p(x)/q(x))$ is bounded by the sum of two positive terms, as 
\begin{equation} \label{eq:lsirestated} D( p \| V) \leq \frac{1}{c K}  \left(  D( \psh \| p) + \left( \log \frac{1}{K} - 1 + K \right)
\right).
\end{equation}
\item If $V = \Pi_{\lambda}$ and $\lambda$ is the mean of $p$, then $\psh$ is the 
size-biased version of $p$ (see for example \cite{johnson24}), $c=1/\lambda$ 
and $K = 1$. We recover the fact that
\begin{equation} \label{eq:wu} D( p \| V) \leq \lambda  D( \psh \| p), \end{equation}
which is a log-Sobolev inequality of Wu \cite[Theorem 1.1]{wu}, reproved more directly in \cite{yu2}
(see also \cite[Corollary 2.4]{chafai}). The relationship
between Theorem \ref{thm:lsi} and \eqref{eq:wu}
 is the same as between the Bakry-\'{E}mery log-Sobolev inequality (Theorem \ref{thm:contlsi}) and the original result of Gross
\cite{gross}.
\item  The RHS of \eqref{eq:lsi} can be understood to be 
\begin{equation} \label{eq:chafai} \frac{1}{c} \sum_{x=0}^\infty V(x) 
A^{\Phi}( f(x), f(x+1)- f(x)), \end{equation}
 where $A^{\Phi}(u,v) = \Phi(u+v) - \Phi(u) - \Phi'(u) v$ is the $A$-transform of the function
$\Phi(u) := u \log u$, as introduced by Chafa\"{i} in \cite{chafai}. This allows us to understand the fact that
Theorem \ref{thm:lsi} reduces to \cite[Corollary 2.4]{chafai} in the case where $V = \Pi_{\lambda}$.
\item Using the bound
$ \log \left( f(x+1)/f(x) \right) \leq  f(x+1)/f(x) - 1$ we deduce that if $V$ satisfies the $c$-log-concavity condition then:
\begin{eqnarray}
\Ent_{V}(f) & \leq & \frac{1}{c} \sum_{x=0}^\infty V(x) \frac{(f(x+1)-f(x))^2}{f(x)}, \label{eq:bltype}
\end{eqnarray}
Hence taking $V = \Pi_\lambda$ and $c= 1/\lambda$  we see Theorem \ref{thm:lsi} generalizes and strengthens 
the log-Sobolev inequality of Bobkov and Ledoux \cite[Corollary 4]{bobkov3}.
\item In the spirit of \cite{johnson11} the RHS of \eqref{eq:bltype} is a form of scaled Fisher information,
equalling
$$ 
\frac{1}{c} \sum_{x=0}^\infty p(x) \left( \frac{ p(x+1) V(x)}{ p(x) V(x+1)} - 1 \right)^2
= \frac{1}{c} \sum_{x=0}^\infty p(x) \left( \frac{  \psh(x)}{K p(x)} - 1 \right)^2,$$
where we may interpret the ratio $ \psh(x)/(K p(x))$ as a scaled score function.
\item \label{it:sharp}
Theorem \ref{thm:lsi} is sharp; equality is achieved in \eqref{eq:lsi} when $V = \Pi_{\lambda}$
for any value of $a$ in $f(x) = \exp(ax + b)$, or equivalently in \eqref{eq:lsirestated} for $p = \Pi_{\mu}$. To verify this, note that 
$p(x)/V(x) = \exp(\lambda - \mu)
(\mu/\lambda)^x$,
so the LHS of \eqref{eq:lsirestated} becomes $\lambda - \mu + \mu \log(\mu/\lambda)$.  Further, 
$c = 1/\lambda$, $\psh = p$
and $K = \lambda/\mu$, the RHS of \eqref{eq:lsirestated} is $\mu ( \log(\mu/\lambda) - 1 + \lambda/\mu)$ and
equality holds.
\item \label{it:modlsi}
Further, Theorem \ref{thm:lsi} strengthens the log-Sobolev inequality of  Caputo et al \cite{caputo} who showed that (under the same
condition) the modified log-Sobolev inequality in the sense of \cite{bobkov11} and Equation \eqref{eq:modlsi} holds:
\begin{equation} \label{eq:lsicap} \Ent_V(f) \leq \frac{1}{c} \sum_{x=0}^\infty V(x) 
(f(x+1)  - f(x)) \log \left( \frac{ f(x+1)}{f(x)} \right). \end{equation}
The expression \eqref{eq:lsicap} is a symmetrized version of \eqref{eq:lsi}, with its RHS equal to the
RHS of \eqref{eq:lsi} plus a similar-looking term (which is  again positive, as before), namely
\begin{equation} \label{eq:lsidiff} \frac{1}{c} \sum_{x=0}^\infty V(x) f(x)  \left( \log \left( \frac{ f(x)}{f(x+1)} \right) - 1
+ \frac{f(x+1)}{f(x)} \right) . \end{equation}
\item Again, to consider sharpness; \cite{caputo} shows  \eqref{eq:lsicap} is sharp, in the weaker sense that the constant
cannot be improved in general. However, equality only holds in \eqref{eq:lsicap} 
for $f(x) = \exp(a x + b)$ in the limit as 
$a \rightarrow \infty$ (the term \eqref{eq:lsidiff} vanishes in this limit).
\item
In the case $V = \Pi_{\lambda}$, the  RHS of \eqref{eq:lsi} is strongly reminiscent of \cite[Corollary 7]{bobkov3},
though that result is expressed in terms of the modulus of differences of functions and their logarithms, and
 is only sharp for $f(x) = \exp(ax+ b)$ where $a \geq 0$.
\end{enumerate}
\end{remark}
In Section \ref{sec:consequences} we discuss some consequences of Theorem \ref{thm:lsi}, including concentration of measure and hypercontractivity results
The assumption that $V$ has support the whole of $\Z_+$ can be relaxed by a perturbation argument (see Remark
\ref{rem:general} below). However, making this assumption simplifies the exposition of the paper.

In a standard way, we consider $f = (1+\epsilon g)$, and let $\epsilon \rightarrow 0$ in \eqref{eq:lsi}. The LHS behaves like $(\epsilon^2/2) \var_V(g)$, and the RHS
like $(\epsilon^2/2) (\frac{1}{c} \sum_{x=0}^\infty V(x) \Delta g(x)^2)$, so that as  expected, the log-Sobolev inequality
Theorem \ref{thm:lsi} implies a  Poincar\'{e} inequality    \eqref{eq:poincare}. 
This Poincar\'{e} inequality
can be proven independently, and  is in fact equivalent to  the slightly weaker integrated BE($c$) assumption,
with the same constant
(see Section
\ref{sec:proofpoinc}).
\begin{theorem}[Poincar\'{e} inequality] \label{thm:poincare}
Fix probability mass function $V$, whose support is the whole of the positive integers $\Z_+$.  Then 
for any function $f$:
\begin{equation} \label{eq:poincare} \var_V(f) \leq \frac{1}{c} \sum_{x=0}^\infty V(x) \Delta f(x)^2, \end{equation}
if and only if $V$ 
satisfies the integrated BE($c$) condition  (Condition \ref{cond:dbec}).
\end{theorem}
In the Poisson case where $V = \Pi_{\lambda}$, taking $c = 1/\lambda$ 
we recover the Poincar\'{e} inequality of Klaasen \cite{klaasen}. In general, as discussed in Section 
\ref{sec:proofpoinc}, Theorem \ref{thm:poincare} is comparable to a Poincar\'{e} inequality
proved under similar conditions by very different methods in \cite{johnson24}.

When $V$ has finite support, we may  adapt the Markov chain, and choose a different $L_V$ with
invariant distribution $V$. The correct statement of the Poincar\'{e} inequality in this context may be  in the
spirit of \cite{johnson29} (where we adapt the form of the derivative operator used). This remains a topic for 
future research.

We briefly describe some related work in the literature.
The proof of Theorem \ref{thm:poincare} mirrors the type of argument given for a range of discrete systems, including a class of
Markov dynamics that includes Kawasaki dynamics, by Boudou, Caputo, Dai Pra and Posta \cite{boudou}. The work of Caputo, Dai Pra and Posta \cite{caputo}
was developed by Fathi and Maas \cite{fathi}, building on a Markov chain-based construction of Ricci curvature on a discrete space introduced by Erbar and Maas \cite{erbar} (see
also \cite{mielke}). In particular, \cite[Theorem 1.5]{erbar} showed that Poincar\'{e} and  modified log-Sobolev inequalities (in the form of \eqref{eq:modlsi}) hold assuming a 
bound on their form of Ricci curvature. 
A form of the $c$-log-concavity condition was used by Joulin \cite{joulin} (along with another form of curvature condition), and was used to derive concentration of measure bounds in the context of birth and death processes.
\section{Bakry--\'{E}mery calculus for continuous spaces} \label{sec:contbakryemery}
We  briefly discuss the classical Bakry-\'{E}mery calculus for continuous spaces -- see \cite{bakry2} for a clear and detailed review.  Since Theorem \ref{thm:lsi} considers measures supported on $\Z_+$, we 
restrict our description to measures on $\re$, although this theory holds in considerably 
greater generality.
The key is a  second-order differential operator $L$,  self-adjoint with respect to  reference measure $d\mu$, which allows the creation of the so-called carr\'{e} du champ operator $\Gam{1}$,
and the iterated operator $\Gam{2}$.
\begin{definition} \label{def:gam1} For any functions $f$ and $g$, write
\begin{eqnarray} 
\Gam{1}(f,g)  &= & \frac{1}{2} \left[ L(f g) - f L g - g L f \right]    \label{eq:gam1}  \\
\Gam{2}(f,g) & = & \frac{1}{2} \left[ L \left( \Gam{1}(f , g) \right) -  \Gam{1} (f, L g) -  \Gam{1}(g, L f) \right]   
 \label{eq:gam2} 
\end{eqnarray}
\end{definition}
The central definition in the theory is the following, which was introduced in \cite{bakry}.
\begin{condition}[Bakry-\'{E}mery condition] \label{cond:be}
 We say that the Bakry-\'{E}mery condition holds with constant $c$ if  for all functions $f$:
\begin{equation} \label{eq:be} \Gam{2}(f,f) \geq c \Gam{1}(f,f).\end{equation}
\end{condition}
One key example is the following, which  motivates the $c$-log-concavity property, Condition \ref{cond:clc}.
It
simplifies further if $U = c x^2/2$ and hence $\mu$ is Gaussian
with variance $1/c$.  
\begin{example}[e.g. \cite{guionnet}, Exercise 4.18] \label{ex:U2bound}
For function $U(x)$, take $d \mu(x) = \exp(-U(x)) dx/\ZZ$. Write 
$$L f(x) = f''(x) - U'(x) f'(x) = \exp(U(x)) \biggl( f'(x) \exp(-U(x)) \biggr)'.$$
It is simple to verify that (for well-behaved $U$, including those satisfying $U''(x) \geq c$, as assumed below),
 $L$ is self-adjoint with respect to $\mu$, that
$\Gam{1}(f,g) = f' g'$, and 
\begin{equation} \label{eq:gamma2val}
\Gam{2}(f,g) 
= f''(x) g''(x) + U''(x) f'(x) g'(x).
\end{equation}
If we assume that $U''(x) \geq c$ then $\Gam{2}(f,f) = f''(x)^2 + U''(x) f'(x)^2 \geq c f'(x)^2 = c \Gam{1}(f,f)$, and the Bakry-\'{E}mery condition \ref{cond:be} holds with constant $c$.
\end{example}

\begin{remark}
As discussed in \cite{guionnet}, the $\Gamma_1$
operator satisfies a product rule of the form:
\begin{equation} \label{eq:gamma1prodrule}
\Gamma_1(f, gh) = \Gamma_1(f,g) h + \Gamma_1(f,h) g.
\end{equation}
As a result (see for example \cite[Lemma 4.12]{guionnet}),  for any well-behaved function $v$, the $\Gamma_1$
operator satisfies a chain rule of the form
\begin{equation} \label{eq:gamma1chainrule}
\Gamma_1( v(f), g) = v'(f) \Gamma_1( f, g),
\end{equation}
which is a key
 reason that the Bakry-\'{E}mery theory applies in  the continuous case.
\end{remark}
 
We state two results  which arise from the Bakry--\'{E}mery calculus, as first described in
\cite{bakry} and reviewed and extended since by a variety of authors. For example, taking $U(t) = t^2$ in \cite[Proposition 5]{bakry} we deduce (see also 
\cite[Proposition 4.8.1]{bakry2}):
\begin{theorem} If 
the Bakry-\'{E}mery condition (Condition \ref{cond:be}) holds with constant $c$ then the Poincar\'{e} inequality holds with constant
$1/c$; that is for any function $f$,
$$ \var_\mu(f) \leq \frac{1}{c} \int  \Gam{1}(f,f)(x) d\mu(x).$$
\end{theorem}
Similarly \cite[Theorem 1]{bakry} (see also \cite[Proposition 5.7.1]{bakry2}) gives that:
\begin{theorem} \label{thm:contlsi}
If 
the Bakry-\'{E}mery condition  (Condition \ref{cond:be}) 
 holds with constant $c$ then the logarithmic Sobolev inequality holds with constant
$1/c$; that is for any function $f$ with positive values:
$$ \Ent_\mu(f) \leq \frac{1}{2c} \int  \frac{\Gam{1}(f,f)(x)}{f(x)} d\mu(x).$$
\end{theorem}
If $\mu$ is Gaussian with variance $\sigma^2$, since (as discussed in Example \ref{ex:U2bound})
we take $c=1/\sigma^2$, and the RHS becomes the standardized Fisher information $\int f'(x)^2/f(x) d\mu(x)$, we 
recover the original log-Sobolev inequality of Gross \cite{gross} (see also Stam \cite{stam}).
\section{Birth and death Markov chain} \label{sec:discretemc}
Fix a probability mass function $V$ supported on the whole of $\Z_+$. As in \cite{caputo}, we construct a  birth and death Markov chain with
invariant distribution $V$.  In \cite{caputo} more general upwards jump rates are considered, but this construction is 
sufficient for our purposes.
\begin{definition} \label{def:bdmc}
Define the birth and death Markov chain with  upward jumps rate equal to 1, and downward jump rate at $x$ equal to $V(x-1)/V(x)$.
Equivalently,  define the $Q$-matrix:
\begin{equation} \label{eq:Qmatrix} Q := \left( \begin{array}{ccccc} 
-1 & 1 & 0 & 0 &  \ldots \\
\frac{V(0)}{V(1)} &  - \frac{V(0)}{V(1)} - 1 & 1 & 0 & \ldots \\
0 & \frac{V(1)}{V(2)} &  - \frac{V(1)}{V(2)} - 1 & 1 &  \ldots \\
& \vdots & \vdots & \vdots & \\ 
\end{array} \right). \end{equation}
We consider evolution of probability mass functions by $p_t := p \exp(t Q) $, so that for any $x$:
\begin{equation} \label{eq:probevol}
\frac{\partial}{\partial t} p_t(x) = p_t Q = p_t(x-1) - \left(1 + \frac{V(x-1)}{V(x)} \right) p_t(x) + \frac{ V(x)}{V(x+1)} p_t(x+1)
= L_V^* p_t(x),
\end{equation}
using the notation of Definition \ref{def:lv}.
\end{definition}
\begin{example} \label{ex:vpoisson}
 If $V$ is Poisson $\Pi_{\lambda}$, then Equation \eqref{eq:probevol} becomes
$$ \frac{\partial}{\partial t} p_t(x) = p_t Q = p_t(x-1) - \left(1 + \frac{x}{\lambda} \right) p_t(x) + \frac{ (x+1)}{\lambda}
p_t(x+1),$$
as in \cite[Equation (14)]{johnson21}, giving the evolution of the $M/M/\infty$ queue. In \cite{johnson21}, the action of
this Markov chain was used to prove the maximum entropy property of the Poisson distribution, under the ultra-log-concavity
condition (Condition \ref{cond:ulc} below). 
\end{example}

Writing vector $\vc{V} = (V(0), V(1), V(2), \ldots )$ the $\vc{V} Q = 0$, so
$V$ is indeed the invariant distribution of this Markov chain. Indeed, the Markov chain satisfies the detailed balance condition,
and hence is reversible.
Further, since $V$ is supported on the whole of $\Z_+$, the Markov chain is irreducible, and
we deduce that this invariant measure is unique, meaning that the probabilities $p_t(x) \rightarrow V(x)$
as $t \rightarrow \infty$. Since the rate of upward jumps is constant, the chain is non-explosive, since the expected time to
reach $\infty$ is $\sum_{x=0}^\infty 1/Q_{x ; x+1} = \sum_{x=0}^\infty 1 = \infty$.

In fact, here it is more useful to consider the evolution of functions.
\begin{definition} Given a function $f$, consider the sequence of functions $f_t$ evolving as $\exp(t Q) f$, so that 
\begin{equation} \label{eq:funcevol}
\frac{\partial}{\partial t} f_t(x) = Q f_t(x) = f_t(x+1) - f_t(x) - \frac{V(x-1)}{V(x)} (f(x) - f(x-1)) = L_V f_t(x),
\end{equation}
where $L_V$ is the operator defined in Definition \ref{def:lv}.
\end{definition}
Next we give a result which allows us to prove the equivalent of  Example \ref{ex:U2bound} above.
\begin{lemma} \label{lem:Laction}
Observe that for any functions $f$ and $g$,  rearrangement gives that $L_V$ is self-adjoint with respect to $V$ where, writing
 $\Delta f(x) = f(x+1) -f(x)$,
\begin{equation} \label{eq:fgdiff}
\sum_{x=0}^\infty V(x) f(x)  L_V g(x) = \sum_{x=0}^\infty V(x) L_V f(x)  g(x) =  - \sum_{x=0}^\infty V(x) \Delta f(x) \Delta g(x). \end{equation}
\end{lemma}
\begin{proof}
This follows by adjusting the index of summation since
\begin{eqnarray*}
\lefteqn{ 
\sum_{x=0}^\infty V(x) f(x)  L_V g(x) } \\
 & = &
\sum_{x=0}^\infty V(x) f(x) \left( g(x+1) - g(x)  - \frac{V(x-1)}{V(x)} \left( g(x) - g(x-1) \right) \right)  \\
& = & \sum_{x=0}^\infty V(x) f(x) \left( g(x+1) - g(x) \right) - \sum_{x=0}^\infty V(x) f(x+1) \left( g(x+1) - g(x) \right) , 
\end{eqnarray*}
and the result follows.
\end{proof}
\section{Integrated Bakry-\'{E}mery condition} \label{sec:dbec}
Given the operator $L_V$, we define the $\GamV{1}$ and $\GamV{2}$ operators induced by it  in the standard way introduced
by \cite{bakry}.
\begin{definition} \label{def:gamv1} For any functions $f$ and $g$, write
\begin{eqnarray} 
\GamV{1}(f,g)  &= & \frac{1}{2} \left[ L_V(f g) - f L_V g - g L_V f \right]    \label{eq:gamv1} \\
\GamV{2}(f,g) & = & \frac{1}{2} \left[ L_V \left( \GamV{1}(f , g) \right) -  \GamV{1} (f, L_V g) -  \GamV{1}(g, L_V f) \right]    \label{eq:gamv2} 
\end{eqnarray}
\end{definition}
We next introduce the Integrated Bakry-\'{E}mery condition; note that in contrast to the classical Bakry-\'{E}mery condition 
(Condition \ref{cond:be}) we only require control of the average (with respect to $V$) of $\GamV{2}$ and $\GamV{1}$, not
pointwise control.
\begin{condition}[Integrated BE($c$)]
\label{cond:dbec}
We say that probability mass function $V$ satisfies the integrated BE($c$) condition if for all functions $f$:
\begin{equation} \label{eq:dbec} \sum_{x=0}^\infty V(x) \GamV{2}(f,f)(x) \geq c \sum_{x=0}^\infty V(x) \GamV{1}(f,f)(x) \end{equation}
\end{condition}

\begin{proposition} \label{prop:clogcgivesdbec}
  For any $f$ and $g$, writing  $L f(x) = f(x+1) - 2 f(x) + f(x-1)$  we deduce:
\begin{eqnarray} 
 \sum_{x=0}^\infty V(x) \GamV{1}(f,g)(x) & = & \sum_{x=0}^\infty V(x) (f(x+1) - f(x))(g(x+1) - g(x)), \label{eq:gamv1sum} \\
\sum_{x=0}^\infty V(x) \GamV{2}(f,g)(x) & =  &  \sum_{x=0}^\infty  V(x)  Lf(x+1) Lg(x+1)  \nonumber \\
&  & +  \sum_{x=0}^\infty  V(x) \EE(x) (f(x+1) - f(x))(g(x+1) - g(x)) . \label{eq:gamv2sum}
\end{eqnarray}
Hence,  if $V$ is $c$-log-concave (if $\EE(x) \geq c$ for all $x$)
then the integrated BE($c$) condition holds.
\end{proposition}
\begin{proof}
Observe that,  the $\EE$ term naturally emerges here and defines
a curvature term, since \eqref{eq:diffLV} expresses the difference between two adjacent derivatives:
\begin{eqnarray} \label{eq:diffLV}
L_V f(x+1) - L_V f(x) & = &  Lf(x+1)  - Lf(x) \frac{V(x-1)}{V(x)}  - \EE(x) (f(x+1) - f(x)).\;\; 
\end{eqnarray}
Using Lemma \ref{lem:Laction}, since $\sum_{x=0}^\infty V(x) L_V h(x) 
= 0$ for any function $h$, we know 
\begin{eqnarray}
\sum_{x=0}^\infty V(x) \GamV{1}(f,g)(x) = - \sum_{x=0}^\infty V(x) f(x) L_V g(x), \label{eq:usual}
\end{eqnarray}
and \eqref{eq:gamv1sum} follows by \eqref{eq:fgdiff}.
Multiplying by $V(x)$ and summing, we recover \eqref{eq:gamv1sum} (as suggested by Lemma \ref{lem:Laction}). 
Using \eqref{eq:gamv1sum}, similarly we know that $\sum_{x=0}^\infty V(x) \GamV{2}(f,g)(x)$ equals
\begin{eqnarray}
\lefteqn{
 - \frac{1}{2} \sum_{x=0}^\infty V(x) \GamV{1}(f, L_V g)(x)  - \frac{1}{2} \sum_{x=0}^\infty V(x) \GamV{1}(L_V f, g)(x)}  \label{eq:equalterms} \\
& = & - \sum_{x=0}^\infty V(x) (g(x+1) - g(x)) \left( L_V f(x+1) - L_V f(x) \right). \label{eq:tosubin} \\
& = & - \sum_{x=0}^\infty V(x) (g(x+1) - g(x)) Lf(x+1) +  \sum_{x=0}^\infty V(x-1) (g(x+1)-g(x)) Lf(x)\;\; \nonumber \\
& & \sum_{x=0}^\infty V(x) \EE(x) (g(x+1) - g(x)) (f(x+1) - f(x)) \nonumber \\
& = & \sum_{x=0}^\infty V(x) \left[ Lf(x+1) Lg(x+1) + \EE(x) (g(x+1) - g(x)) (f(x+1) - f(x)) \right] \label{eq:done}
\end{eqnarray}
where \eqref{eq:tosubin} follows by \eqref{eq:gamv1sum}, since the two terms in \eqref{eq:equalterms} are both equal (as \eqref{eq:usual} shows that as usual, they can both be expressed as 
$\frac{1}{2} \sum_{x=0}^\infty V(x) L_V f(x) L_V g(x)$).
The final result \eqref{eq:done} follows on relabelling, having substituted \eqref{eq:diffLV} in the second term of \eqref{eq:tosubin}.
 \end{proof}

\begin{remark} \label{rem:prodrule} 
Using \eqref{eq:gamv1sum} we deduce that $\GamV{1}$ only satisfies a modified form of the product rule in 
\eqref{eq:gamma1prodrule}. That is since $g(x+1) h(x+1) - g(x) h(x) = h(x+1) (g(x+1) - g(x)) + g(x) (h(x+1)-h(x))$ we know
that
\begin{eqnarray*}
 \sum_{x=0}^\infty V(x) \GamV{1}(f,g h)(x)   
& = & \sum_{x=0}^\infty V(x) (f(x+1) - f(x)) (g(x+1) - g(x))  h(x+1) \\
& & + \sum_{x=0}^\infty V(x) (f(x+1) - f(x)) (h(x+1)-h(x)) g(x)
\end{eqnarray*}
\end{remark}

\section{The $c$-log-concavity condition} \label{sec:dbeulc}
The $c$-log-concavity property (Condition \ref{cond:clc})  corresponds to the bound $U''(x) \geq c$ discussed in 
Example \ref{ex:U2bound}. Condition \ref{cond:clc} was introduced as Assumption A in \cite{caputo}, who showed that it is implied
by the ultra-log-concavity condition of Pemantle \cite{pemantle} and Liggett
\cite{liggett}:
\begin{condition}[ULC]
\label{cond:ulc} If a probability mass function $V$ has the property that $V/\Pi_{\lambda}$ is a log-concave sequence, then we
say that $V$ is ultra-log-concave (ULC).
\end{condition}
\begin{lemma}[\cite{caputo}, Section 3.2] \label{lem:ulcgivesclogc} If $V$ is ULC, then it is $c$-log-concave, with $c = V(0)/V(1)$. \end{lemma}

Notice that if $U$ and $V$ are probability mass functions then 
$ \diy \frac{ (U \star V)(1)}{ (U \star V)(0)} = \frac{ U(1)}{U(0)} + \frac{ V(1)}{V(0)}, \eiy$
where $(U \star V)$ represents the convolution.
In the light of Lemma \ref{lem:ulcgivesclogc} this suggests the conjecture that if $U$ and $V$ are $c$-log-concave with
constants $c_U$ and $c_V$ respectively, then $(U \star V)$ is $c$-log-concave with constant $\geq (1/c_U + 1/c_V)^{-1}$.
(Recall that Walkup \cite{walkup} proved a result which implies that if $U$ and $V$ are ULC, then so is $(U \star V)$.)

We discuss probability
mass functions $V$ for which Condition \ref{cond:clc} is satisfied. While Theorem \ref{thm:lsi} requires that $V$ has support the whole of $\Z_+$, it is 
still instructive to take $V$ with finite interval support (see Remark \ref{rem:general}).
\begin{example} \label{ex:clogc} \mbox{ } 
\begin{enumerate}
\item If $V = \Pi_\lambda$ is Poisson, then since $V(x)/V(x+1) = (x+1)/\lambda$, we know that $\EE(x) \equiv 1/\lambda$, so
$V$ is $c$-log-concave (with equality), with $c = 1/\lambda$.
\item By Lemma \ref{lem:ulcgivesclogc}, the probability mass function $V$ of the sum of independent Bernoulli variables with mean $p_i$, is
$c$-log-concave with $c = \left( \sum_j p_j/(1-p_j) \right)^{-1}$.
\item If $V(x) = \binom{n+x-1}{x} p^x (1-p)^n$ is negative binomial, then direct calculation gives
$\diy \EE(x) = \frac{(n-1)}{p( n+x) (n+x-1)}, \eiy$
which tends to zero as $x \rightarrow \infty$. Hence $V(x)$ is only $c$-log-concave with $c = 0$.
\end{enumerate}
\end{example}

One final remark is that no mass function with mean $\ep V$ can be $c$-log-concave for $c > 1/(\ep V)$. Hence the value $1/\lambda$
found for $\Pi_\lambda$ in Example \ref{ex:clogc} is an extreme one.
\begin{lemma} \label{lem:cbound}
 If $V$ is $c$-log-concave, then $c \leq 1/(\ep V)$. \end{lemma}
\begin{proof}
Since $\EE$ is a finite difference, we sum the collapsing sum to obtain
$$ \frac{V(x)}{V(x+1)} = \sum_{y=0}^{x} \left( \frac{V(y)}{V(y+1)} - \frac{V(y-1)}{V(y)} \right)
\geq (x+1) c,$$
by assumption. Rearranging and summing we obtain that
$$ 1 = \sum_{x=0}^\infty V(x) \geq \sum_{x=0}^\infty  (x+1) V(x+1)  c = c (\ep V),$$
and the result follows.
\end{proof}

Note further that in some settings it may be natural to assume that $\EE(x)$ is increasing in $x$.  Direct substitution shows that
this is equivalent to the property that
\begin{equation} \label{eq:increase}
V(x)^2 V(x-1)  - 2 V(x-1)^2 V(x+1) + V(x+1) V(x) V(x-2) \geq 0, \mbox{\;\;\;\; for all $x \geq 0$.}
\end{equation}
In \cite{johnson34}, this property (referred to there as `Property $C_1(k)$') is shown by induction to hold when $V$ is
the probability mass function of the sum of independent Bernoulli variables, and it is natural to assume that \eqref{eq:increase}
holds in a more general setting than this.

\section{Proof of the log-Sobolev inequality, Theorem \ref{thm:lsi}} \label{sec:prooflsi}
\begin{proof}[Proof of  Theorem \ref{thm:lsi}]
Given a fixed probability mass fucntion $V(x)$ and a function $f$ with $\sum_{x=0}^\infty V(x) f(x) = \mu_{V,f}$, we consider function $f_t$
evolving as \eqref{eq:funcevol}, that is with $f_0 \equiv f$ and 
$$ \frac{\partial}{\partial t} f_t(x) =  L_V f_t(x).$$
Note that, by ergodicity, $\lim_{t \rightarrow \infty} f_t(x) = \sum_{x=0}^\infty V(x) f(x) = \mu_{V,f}$.
We consider the function 
\begin{equation}  \Theta(t) = \sum_{x=0}^\infty V(x) f_t(x) \log f_t(x),\end{equation}
and  obtain that (as in \cite{caputo}):
\begin{eqnarray}
\Theta'(t) & = & \sum_{x=0}^\infty V(x) L_V f_t(x) \log f_t(x) + \sum_{x=0}^\infty V(x) f_t(x) \frac{L_V f_t(x)}{f_t(x)} \nonumber \\
& = & - \sum_{x=0}^\infty V(x) \left( f_t(x+1) - f_t(x) \right) \left( \log f_t(x+1) - \log f_t(x) \right). \label{eq:thetaprime}
\end{eqnarray}
This follows by cancellation, since $\sum_{x=0}^\infty V(x) L_V h(x) = 0$ for any $h$,
and by taking $f=f_t$ and $g =\log f_t$ in \eqref{eq:fgdiff}.
Since both terms in brackets in \eqref{eq:thetaprime}
have the same sign, we  conclude that $\Theta'(t) = -\sum_{x=0}^\infty V(x) \GamV{1}(f_t, \log f_t)(x) \leq 0$ (this is the term arising in \eqref{eq:lsicap}).
 However the absence of a chain rule of the type \eqref{eq:gamma1chainrule}
means that we cannot write it in a form where Condition \ref{cond:dbec} can be directly applied.
However, we calculate a further derivative by hand.

In fact, we consider the derivative of a related term, which we think of as only part of the expression for $\Theta'(t)$. That
is, we write
$$ \psi(t) = \sum_{x=0}^\infty V(x) \left( f_t(x+1) \log \left( \frac{ f_t(x+1)}{f_t(x)} \right) - f_t(x+1) + f_t(x) \right).$$
Using the fact that for functions $g$ and $h$, $(g \log(g/h) - g)' = g' \log(g/h) - g h'/h$, 
by relabelling in the usual way we  deduce that
\begin{eqnarray}
\psi'(t) & = & \sum_{x=0}^\infty V(x)  \left( L_V f_t(x+1) - L_V f_t(x) \right) \log \left( \frac{ f_t(x+1)}{f_t(x)} \right)  \label{eq:term1} \\
& & +  \sum_{x=0}^\infty V(x) L_V f_t(x) \left(  \log \left( \frac{f_t(x+1)}{f_t(x)} \right) - \frac{f_t(x+1)}{f_t(x)} \right) \label{eq:term2}
\end{eqnarray}
By taking $g(x) = \log f_t(x)$ and $ f(x) = f_t(x)$ in \eqref{eq:tosubin}, we deal with \eqref{eq:term1}, and by taking 
$f (x)= f_t(x)$ and $ \diy g(x) =  \log \left( \frac{f_t(x+1)}{f_t(x)} \right) - \frac{f_t(x+1)}{f_t(x)} \eiy$ in \eqref{eq:fgdiff}, we
deal with \eqref{eq:term2}. Adding the results of these manipulations together, we deduce that
\begin{eqnarray}
\psi'(t) & = & - \sum_{x=0}^\infty V(x) \EE(x) (f_t(x+1) - f_t(x) )  \log \left( \frac{ f_t(x+1)}{f_t(x)} \right) \label{eq:2ndder} \\
& & \;\;\; + \sum_{x=0}^\infty V(x) f_t(x+1) w \left( \frac{f_t(x) f_t(x+2)}{f_t(x+1)^2} ; \frac{f_t(x)}{f_t(x+1)} \right), \label{eq:2ndder2}
\end{eqnarray}
where $w(U; s) = -(U/s - 1) \log U + (1-U)(1- 1/s)$.
  Lemma \ref{lem:tech} below gives that the term  \eqref{eq:2ndder2} is negative. (Note that this term is zero if $f_t(x) = \exp(a x + b)$,
which contributes to the sharpness result discussed in Remark \ref{rem:compare}).
 Further, by assumption, we  bound \eqref{eq:2ndder} from above
on replacing $\EE(x)$ by $c$. In other words, we deduce by comparison with \eqref{eq:thetaprime} that
$\psi'(t) \leq c \Theta'(t)$, or that $\diy (-\Theta'(t)) \leq \frac{1}{c} (-\psi'(t)) \eiy$.
We  deduce 
\begin{eqnarray*}
\Ent_V(f) =  \Theta(0) - \Theta(\infty) 
& = & \int_{0}^{\infty} -  \Theta'(t) dt \\
& \leq & \frac{1}{c} \int_0^\infty (-\psi'(t)) dt =  \frac{1}{c} \psi(0) 
\end{eqnarray*}
and the result follows.
\end{proof}
\begin{lemma} \label{lem:tech}
The function $w(U; s) = -(U/s - 1) \log U + (1-U)(1- 1/s) \leq 0$ for all $s, U \geq 0$,
with equality if and only if $U = 1$.
\end{lemma}
\begin{proof}
For fixed $s$, we observe that $w(1; s)  = 0$, that $\diy \frac{\partial}{\partial U} w(U;s) \big|_{U=1} = 0 \eiy$ and $w(U; s)$ is a strictly concave function, 
since $ \diy \frac{\partial^2}{\partial U^2} w(U; s) = - \frac{s + U}{s U^2} \eiy$.
\end{proof}

\begin{remark} \label{rem:general} If $V$ has support on a finite interval,  a version of Theorem
\ref{thm:lsi} should still hold, at least for a class of functions $f$. In brief, define $V_\epsilon := V \star \Pi_{\epsilon}$ to be the convolution of $V$ with a Poisson mass function
of mean $\epsilon$. If $V$ is $c$-log-concave, then for any given $\delta$, the $V_{\epsilon}$ will be $(c-\delta)$-log-concave
for $\epsilon$ sufficiently small. Hence, we can apply Theorem \ref{thm:lsi} to $V_{\epsilon}$ (which is supported on the whole
of $\Z_+$ as required) to obtain a bound on $\Ent_{V_\epsilon}(f)$.

 Further, by continuity arguments using dominated convergence  $\Ent_{V_\epsilon}(f)$ will converge to $\Ent_V(f)$ for well-behaved $f$, and the resulting
upper bound will also converge.  However, we omit further discussion of this and the correct class of $f$ to use for the sake 
of brevity.
\end{remark}

\section{Proof of the Poincar\'{e} inequality, Theorem \ref{thm:poincare}} \label{sec:proofpoinc}
We show that the Poincar\'{e} inequality is equivalent to the integrated BE($c$) condition, using a standard argument (see for example
 Proposition 4.8.3 of \cite{bakry2}).
\begin{proof}[Proof of Theorem \ref{thm:poincare}]
First, we assume the integrated BE($c$) condition, and write $\Lambda(t) = \sum_{x=0}^\infty V(x) f_t(x)^2$.
By Lemma \ref{lem:Laction} and \eqref{eq:gamv1sum}
\begin{eqnarray*}
\Lambda'(t) = 2 \sum_{x=0}^\infty V(x) f_t(x) L_V f_t(x) = -  2 \sum_{x=0}^\infty V(x) (\Delta f_t)(x)^2 = - 2 \sum_{x=0}^\infty V(x) \GamV{1}(f_t, f_t)
\end{eqnarray*}
Similarly, since $\frac{\partial}{\partial t}$ commutes with $L_V$ by the form of the $Q$-matrix in Definition \ref{def:bdmc},
\begin{eqnarray*}
\Lambda''(t) & = & 2 \sum_{x=0}^\infty V(x) L_V f_t(x) L_V f_t(x) + 2 \sum_{x=0}^\infty V(x) f_t(x) L_V^2  f_t(x) \\
& = & 4 \sum_{x=0}^\infty V(x) \left(  L_V f_t(x) \right)^2  
=  4 \sum_{x=0}^\infty V(x) \GamV{2}(f_t, f_t).
\end{eqnarray*}
The integrated BE($c$) condition applied to the function $f_t$ tells us that $\Lambda''(t) \geq -2 c \Lambda'(t)$. This tells us
that
\begin{eqnarray*}
\var_V(f) =  \Lambda(0) - \Lambda(\infty) & = &  \int_{0}^{\infty} -  \Lambda'(t) dt \\
& \leq & \frac{1}{2c} \int_0^\infty \Lambda''(t) dt 
=  \frac{1}{2c} \left( - \Lambda'(0)  \right) 
= \frac{1}{c} \sum_{x=0}^\infty V(x) (\Delta f)(x)^2,
\end{eqnarray*}
and the result follows.

Second, if the Poincar\'{e} inequality holds, we deduce the integrated BE($c$) condition, since without loss of 
generality we can consider
for any $f$ with 
$\sum_x V(x) f(x) = 0$, for which
\begin{eqnarray}
\sum_{x=0}^\infty V(x) \GamV{1}(f,f)(x) & = & - \sum_{x=0}^\infty V(x) f(x) L_V f(x) \label{eq:equiv1} \\
& \leq & \sqrt{ \sum_{x=0}^\infty  V(x) f(x)^2} \sqrt{ \sum_{x=0}^\infty V(x) L_V f(x)^2 } \label{eq:equiv2} \\
& \leq & \sqrt{ \frac{1}{c} \sum_{x=0}^\infty V(x) \GamV{1}(f,f)(x) }
\sqrt{
\sum_{x=0}^\infty V(x) \GamV{2}(f,f)(x) }, \label{eq:equiv3}
\end{eqnarray}
where \eqref{eq:equiv1} follows by  \eqref{eq:usual}, \eqref{eq:equiv2} follows by Cauchy--Schwarz,
  and \eqref{eq:equiv3} follows since
by \eqref{eq:poincare}
 $\sum_{x=0}^\infty V(x) f(x)^2 = \var_V(f) \leq \frac{1}{c} \sum_{x=0}^\infty V(x) \Delta f(x)^2 = \frac{1}{c} 
\sum_{x=0}^\infty V(x)
\GamV{1}(f,f)(x)$.
\end{proof}

Theorem \ref{thm:poincare} shows that if $V$ satisfies the integrated BE($c$) condition, then
the Poincar\'{e} constant of $V$ is  $\leq 1/c$. In comparison \cite[Corollary 2.4]{johnson24}, which was
proved using arguments based
on stochastic ordering and size-biasing, shows that
if $V$ is ULC then the Poincar\'{e} constant of $V$ is less than or equal to $\ep V$.
Lemma \ref{lem:ulcgivesclogc} and Proposition \ref{prop:clogcgivesdbec} show ULC  implies the integrated  BE($c$) condition,
 hence the assumptions of the present paper are weaker than in 
\cite{johnson24}. However, Lemma \ref{lem:cbound} shows that $\ep V \leq 1/c$,  so here we prove 
a weaker bound on the Poincar\'{e} constant. It would be of interest to know if the two approaches can be synthesised, or
if the results are each optimal under their own assumptions.

\section{Consequences of Theorem \ref{thm:lsi}} \label{sec:consequences}

We briefly discuss some results which follow from Theorem \ref{thm:lsi}, including a concentration of measure inequality,
decay of entropy and a form of hypercontractivity.

\subsection{Concentration of measure}
We prove a concentration of measure result  by adapting the argument used to prove \cite[Proposition 10]{bobkov3}, and deduce the following bound:
\begin{proposition} \label{prop:concmeas}
Fix  probability mass function $V$, and suppose that for all functions $f$ with positive values, Equation \eqref{eq:lsi} holds, that is:
\begin{equation} \label{eq:lsi2} \Ent_V(f) \leq \frac{1}{c} \sum_{x=0}^{\infty}
 V(x) f(x+1)  \left( \log \left( \frac{ f(x+1)}{f(x)} \right) - 1
+ \frac{f(x)}{f(x+1)} \right) . \end{equation}
Then, writing $h(s) = (1+s) \log (1+s) - s$,  for any function $g$ with $\sup_{x} |g( x+1) - g(x)| \leq 1$:
\begin{equation}
V \left( \left\{ g \geq \ep_V g + t \right\} \right)  \leq \exp \left( - \frac{ h(c t)}{c} \right).
\end{equation}
\end{proposition}
\begin{proof} Define the function $G(\tau) = \sum_{x=0}^\infty V(x) e^{ \tau g(x)}$, and the related function
$H(\tau) = \left( \log G(\tau) \right)/\tau$. Taking $f(x) = e^{\tau g(x)}$ in \eqref{eq:lsi2} we deduce that:
\begin{eqnarray}
\tau^2 G(\tau) H'(\tau) & = & \tau G'(\tau) - G(\tau) \log G(\tau)  \nonumber \\
& = & \sum_{x=0}^\infty V(x) \tau g(x)  e^{\tau g(x)} - G(\tau) \log G(\tau) \nonumber \\
& = & \Ent_V \left( e^{\tau g} \right) \nonumber \\
& \leq & \frac{1}{c} \sum_{x=0}^\infty V(x) \left[ \tau e^{\tau g(x+1)} (g(x+1) - g(x)) - e^{\tau g(x+1)} + e^{\tau g(x)} \right]
\nonumber \\
& = & \frac{1}{c} \sum_{x=0}^\infty V(x) e^{\tau g(x)} \varphi \left( \tau \Delta g(x) \right), \label{eq:Htbounds}
\end{eqnarray}
where $\varphi(u) = u e^u - e^u + 1 \geq 0$ and $\Delta g(x) = g(x+1) - g(x)$.

Since $\varphi'(u) = u e^u$, which has the same sign as $u$, we know that taking $\tau \geq 0$ and
  for $v \in (-\tau, \tau)$, the
$\varphi( v) \leq \max \left( \varphi(\tau), \varphi(- \tau) \right) = \varphi(\tau)$, where this last inequality 
follows since $\varphi(v) - \varphi(-v)$ is increasing on $v \geq 0$, and hence is $\geq 0$. Using this,
we can rewrite \eqref{eq:Htbounds} in the form $ \tau^2 G(\tau) H'(\tau) \leq  \frac{\varphi(\tau)}{c} G(\tau)$, which
we can integrate to deduce that for any $\sigma \geq 0$:
\begin{equation} \label{eq:lsicompare}
H(\sigma) - H(0) = \int_0^\sigma H'(\tau) d \tau \leq \frac{1}{c} \int_0^\sigma \frac{\varphi(\tau)}{\tau^2} d \tau
= \frac{1}{c} \frac{ e^{\sigma} - \sigma - 1}{\sigma}. \end{equation}
This can be rearranged to give an upper bound on $G(\sigma)$.
As in \cite[Proposition 10]{bobkov3}, we can use a standard Chernoff bounding argument, based on the fact that $H(0) = \ep_V g$ and using Markov's inequality to deduce that
for any $\sigma > 0$:
\begin{eqnarray*}
V \left( \left\{ g \geq \ep_V g + t \right\} \right) \leq \frac{ \ep_V e^{\sigma g}}{e^{\sigma (\ep_V g + t)}} =
 \frac{G(\sigma)}{e^{\sigma (\ep_V g + t)}} \leq \exp \left( \frac{ e^\sigma - \sigma - 1}{c} - \sigma t \right).
\end{eqnarray*}
We  make the optimal choice of $\sigma$ here, that is $\sigma = \log(1+ c t)$, to deduce the result.
\end{proof}
Note this function $h$ commonly occurs in concentration of measure results in different settings, including
Bennett's inequality (see for example \cite[Theorem 9]{raginsky}), work of Houdr\'{e} and co-authors based on the `covariance method' (see for
example \cite[Eq. (1.6)]{houdre2}) and recent work on discrete random variables
 using a tail condition under coupling
\cite[Theorem 3.3]{cook2}.

\begin{remark}
Proposition \ref{prop:concmeas}
 shows that Theorem \ref{thm:lsi} can provide practical improvements to 
results of the form \eqref{eq:bltype}. To be specific, \cite[Proposition 10]{bobkov3} shows that if \eqref{eq:bltype} holds,
then, under the same condition on $\Delta g$:
\begin{equation} \label{eq:tailboundsbl}
V \left( \left\{ g \geq \ep_V g + t \right\} \right)  \leq \exp \left( - \frac{ k(c t)}{c} \right), \end{equation}
where $k(u) = u \log(1+u)/4$. Proposition \ref{prop:concmeas} therefore  strengthens \eqref{eq:tailboundsbl}
 under the $c$-log-concavity
condition, Condition \ref{cond:clc}, since $h(u) \geq 2 k(u)$ for all $u$.
This strengthening comes from the fact that the expression of \eqref{eq:lsicompare} 
is significantly smaller than the bound of $\frac{1}{2c} (e^{2 \sigma} - 1)$ which follows by  the
argument of \cite[Proposition 10]{bobkov3}. Note that \eqref{eq:lsicompare} is sharp, in the sense that equality holds when taking $V = \Pi_\lambda$ and $g (x) = x$,
as follows from the sharpness of Theorem \ref{thm:lsi} discussed in Remark \ref{rem:compare}.\ref{it:sharp}.
\end{remark}

\subsection{Decay of entropy and hypercontractivity }
We briefly discuss how the log-Sobolev inequality, Theorem \ref{thm:lsi}, implies further results for 
related processes, in a standard way.
Motivated by  the paper \cite{yu2}, which considered pure thinning, we consider probability measures evolving as the 
`death' part of the birth and death process. That is, for fixed $V$, we consider probability distributions such that:
\begin{eqnarray}
\frac{\partial}{\partial t} V_t(x) & = &  \alpha_t \left( V_t(x) -  V_t(x-1) \right), \label{eq:vbeh} \\
 \frac{\partial}{\partial t} p_t(x) & = & \alpha_t \left( \frac{V_t(x)}{V_t(x+1)} p_t(x+1) - \frac{V_t(x-1)}{V_t(x)} p_t(x) \right). 
\end{eqnarray}
\begin{proposition} \label{prop:decay}
If $V_t(x)$ satisfies $\EEt(x) \geq c_t$ for all $x$ then 
$$D( p_t \| V_t) \leq D(p \| V) \exp\left( - \int_0^t \alpha_s c_s ds \right).$$
\end{proposition}
\begin{proof}
Writing $K_t = \left( \sum_{x=0}^\infty p_t(x+1) V_t(x)/V_t(x+1) \right)^{-1}$ and
$\psh_t(x) = K_t p_t(x+1) V_t(x)/V_t(x+1)$ relabelling gives:
\begin{eqnarray*}
\lefteqn{ \frac{\partial}{\partial t} D( p_t \| V_t) } \\
& = &  \alpha_t
\sum_{x=0}^\infty \left( \frac{V_t(x)}{V_t(x+1)} p_t(x+1) - \frac{V_t(x-1)}{V_t(x)} p_t(x) \right) \log \left( \frac{p_t(x)}{V_t(x)} \right)
-   \frac{V_t(x) -  V_t(x-1)}{V_t(x)} p_t(x) \\
& = & \alpha_t \left( \sum_{x=0}^\infty \frac{ V_t(x) p_t(x+1)}{V_t(x+1)} \log \left( \frac{ p_t(x) V_t(x+1)}{V_t(x) p_t(x+1)} \right) - 1 + \frac{1}{K_t} \right) \\
& = & - \frac{\alpha_t}{K_t} \left( D( \psh_t \| p_t) + \log \frac{1}{K_t} - 1 + K_t \right) \\
& \leq & - \alpha_t c_t D(p_t \| V_t),
\end{eqnarray*}
where the last inequality follows using the form of the log-Sobolev inequality given
by Equation \eqref{eq:lsirestated}.
\end{proof}
\begin{example} \label{ex:thin}
 Taking $V_t = \Pi_{\lambda(t)}$, a Poisson mass function with mean $\lambda(t)  = \lambda e^{-t}$,
then \eqref{eq:vbeh} holds with $\alpha_t = \lambda(t)$, and we know that $c_t = 1/\lambda(t)$. Hence, $P_t$ becomes
the mass function $P$ thinned by $e^{-t}$ (see \cite{johnson24} for a discussion of this operation), and we can deduce that
\begin{equation} \label{eq:thind} D( p_t \| V_t) \leq D(p \| V) e^{-t}. \end{equation}
\end{example}
We also illustrate Theorem \ref{thm:lsi} by using it to prove a form of hypercontractivity, using a standard argument
(see for example \cite[Theorem 11]{bobkov11} and \cite[Page 246]{bakry2}).
\begin{proposition} \label{prop:hyper}
Consider a sequence of probability measures evolving as in \eqref{eq:vbeh} and a sequence of functions evolving in a related
way:
\begin{eqnarray}
\frac{\partial}{\partial t} V_t(x) & = &  \alpha_t \left( V_t(x) -  V_t(x-1) \right),  \nonumber \\
\frac{\partial}{\partial t} g_t(x) & = & \alpha_t \frac{V_t(x-1)}{V_t(x)} \left( g_t(x) - g_t(x-1) \right).  \label{eq:funcevol2}
\end{eqnarray}
If $V_t$ satisfies the new modified log-Sobolev inequality, Equation \eqref{eq:lsi} with constant $c_t$ then 
writing $q(t) = p \exp\left( - \int_0^t \alpha_s c_s ds \right)$ and $ \| f \|_{U,p} = \left( \sum_{x=0}^\infty U(x) f(x)^p \right)^{1/p}$ then
\begin{equation} \label{eq:hyper}
\| \exp(g ) \|_{V,p} \leq \|  \exp(g_t) \|_{V_t,q(t)}. \end{equation}
(Note that $q(t) \leq p$).
\end{proposition} 
\begin{proof} As in \cite{bakry2, bobkov11}, we consider the functional
$ \Lambda(q,t) := \sum_{x=0}^\infty V_t(x) \exp( q g_t(x)).$
The key is to  express
\begin{eqnarray}
\frac{\partial}{\partial q} \Lambda(q,t) 
& = & \sum_{x=0}^\infty V_t(x) \exp(q g_t(x)) g_t(x) \nonumber \\
& = & \frac{1}{q} \Ent_{V_t}( \exp(q g_t) ) + \frac{1}{q}  \Lambda(q,t) \log \Lambda(q,t) \nonumber 
\end{eqnarray}
and using \eqref{eq:vbeh} and \eqref{eq:funcevol2} to recognise that $\frac{\partial}{\partial t} \Lambda(q,t)$ equals
$$
 \alpha_t \left[ \sum_{x=0}^\infty V_t(x+1)  \left( q \exp(q g_t(x+1)) \left(  g_t(x+1) - g_t(x) \right)
- \exp( q g_t(x+1))  + \exp(q g_t(x))  \right) \right].$$
Taking $f=\exp( q g_t)$ in  \eqref{eq:lsi} we deduce that
\begin{equation} \label{eq:logderivative}
- \partial_q \log \Lambda(q,t) + \frac{1}{\alpha_t c_t q} \partial_t \log \Lambda(q,t)
+ \frac{1}{q} \log \Lambda(q,t) \geq 0.
\end{equation}
Using this, we can consider the behaviour of $u(t) := \log \Lambda( q(t), t)/q(t)$. Taking a derivative with respect to $t$,
using the fact that $q'(t)/q(t) = - c_t \alpha_t$,  we 
obtain that
\begin{eqnarray*}
u'(t) & = & \frac{ q'(t)}{q(t)} \partial_q \log \Lambda(q(t), t) + \frac{1}{q(t)} \partial_t \log \Lambda(q(t),t) - \frac{q'(t)}{q(t)^2}
\log \Lambda(q(t),t) \\
& = & c_t \alpha_t \left( - \partial_q \log \Lambda(q(t), t) + \frac{1}{q(t) c_t \alpha_t} \partial_t \log \Lambda(q(t),t)
+ \frac{1}{q(t)} \log \Lambda(q(t),t) \right) \\
& \geq & 0,
\end{eqnarray*}
where the final inequality follows from \eqref{eq:logderivative}. Since $\exp( u(t)) = \| \exp(g_t) \|_{V_t,q(t)}$, we deduce that the
$q$ norm is increasing as required.
\end{proof}

Note that the definition of $q(t)$  involves the same  exponential expression as Proposition \ref{prop:decay}.

\begin{example}
As in Example \ref{ex:thin}, 
we can consider $V_t = \Pi_{\lambda(t)}$, a Poisson mass function with mean $\lambda(t)  = \lambda e^{-t}$,
and notice that \eqref{eq:funcevol2} is satisfied by the Poisson-Charlier polynomials $c_k(x; \lambda(t))$. Further, since
$c_t \alpha_t = 1$, we take $q(t) = p \exp(-t)$.

We deduce that Proposition \ref{prop:hyper} is sharp, taking $g_t = (x - \lambda(t))/\lambda(t)$ to be the Poisson--Charlier
polynomial of degree 1.  In this case
\begin{eqnarray*}
 \Lambda(q(t), t) & = & \sum_{x=0}^\infty \frac{ \exp(-\lambda(t) \lambda(t)^x}{x!} \exp \left( \frac{ q(t) (x - \lambda(t))}{\lambda(t)} \right) \\ 
& = &  \exp( - q(t) - \lambda(t) + \lambda(t) e^{q(t)/\lambda(t)} )
= \exp( - q(t) C),
\end{eqnarray*}
where $C = 1 + \lambda/p - \lambda/p e^{p/\lambda}$,
using the fact that $\lambda(t)/q(t) \equiv \lambda/p$. Hence the $q(t)$ norm is constant, and Proposition \ref{prop:hyper}
is sharp.
(This sharpness corresponds to the sharpness of the new modified log-Sobolev inequality for functions of the form 
$f(x) = \exp(a x + b)$, as discussed in Remark \ref{rem:compare}.\ref{it:sharp}).
\end{example}

\section{Extension to random variables on $\Z_+^d$}

It would be of considerable interest to extend this work to the more general setting of probability measures on graphs, where
curvature and related issues are topics of active research. For example, \cite{ane2} uses the
Bakry-\'{E}mery $\Gamma$-calculus to deduce log-Sobolev inequalities on the discrete cube, $\Z^d$ and general graphs of
uniformly bounded degree.  The paper \cite{klartag} defines curvature for discrete graphs, and shows that 
controlling this curvature allows results   including Poincar\'{e} and log-Sobolev inequalities to be deduced. Lin and Yau \cite{lin} compare the two forms of
curvature discussed by Joulin \cite{joulin}, in the context of graphs.
 
We briefly describe how the $c$-log-concavity condition, Condition \ref{cond:clc},  extends to
the setting of probability measures on $\Z_+^d$. We deduce an integrated Bakry-\'{E}mery condition, and hence a Poincar\'{e} 
inequality, and explain the issues with proving  a modified log-Sobolev inequality in the form of Theorem \ref{thm:lsi}.

Fix a reference measure $V(\vc{x})$ which is positive for all
$\vc{x} \in \Z_+^d$, and
write $\vc{e}_i$ for the $i$th unit vector. Further,  for all $1 \leq i, j \leq d$ and for a given function $f$
 we define
\begin{eqnarray}
\EEd{i}{j}(\vc{x}) & = &  \frac{ V( \vc{x} + \vc{e}_j - \vc{e}_i)}{V( \vc{x} + \vc{e}_j)}
- \frac{ V( \vc{x} - \vc{e}_i)}{V( \vc{x})},  \label{eq:eeij} \\
L_{ij} f(\vc{x}) & = &  f(\vc{x} + \vc{e}_j) - f(\vc{x} +  \vc{e}_j - \vc{e}_i) -   f(\vc{x}) + f(\vc{x} - \vc{e}_i).
\end{eqnarray}
Notice that these quantities are not symmetric in $i$ and $j$, however  for each $\vc{y}$ and each $c$ we can define a symmetric matrix by
\begin{eqnarray} 
\EEsym{i}{j}(\vc{y})
& := & \frac{ V(\vc{y} - \vc{e}_i)  V(\vc{y} - \vc{e}_j)}{V(\vc{y})} -  V(\vc{y} - \vc{e}_i - \vc{e}_j) - c \II(i =j ) V(\vc{y} - \vc{e}_j). \label{eq:Edef} \\
 & = & V( \vc{y} - \vc{e}_j) \left( \EEd{i}{j}(\vc{y} - \vc{e}_j) - c \II(i = j) \right) \nonumber 
\end{eqnarray}
writing $\II(i = j)$ for the entries of the identity matrix.
Consider a  process which (for all $\vc{x}$ and all $i$) 
jumps from $\vc{x}$  to $\vc{x} + \vc{e}_i$ at rate 1 and
from $\vc{x}$ to $\vc{x} - \vc{e}_i$ at rate $V(\vc{x} - \vc{e}_i)/V(\vc{x})$ (where by convention $V(\vc{y}) = 0$ if any component of
$\vc{y}$ is $-1$). This corresponds to defining 
\begin{equation} \label{eq:lvd}
L_V f(\vc{x}) = \sum_{i=1}^d \left( f( \vc{x} + \vc{e}_i) - f(\vc{x}) \right) - \frac{V(\vc{x} - \vc{e}_i)}{V(\vc{x})}
\left( f(\vc{x}) - f(\vc{x} - \vc{e}_i) \right)
\end{equation}
The key is to observe that an analogue of \eqref{eq:diffLV} holds, that is direct calculation gives that for any $j$:
\begin{eqnarray} \label{eq:diffLVd}
\lefteqn{L_V f( \vc{x}+ \vc{e}_j) - L_V f( \vc{x})}  \nonumber \\ 
& = & \sum_{i=1}^d \left(
 L_{ij} f( \vc{x}+ \vc{e}_i)  - L_{ij} f(\vc{x}) \frac{V( \vc{x}- \vc{e}_i)}{V(\vc{x})}  - \EEd{i}{j}(\vc{x}) 
\left( f( \vc{x} + \vc{e}_j) - f(\vc{x} + \vc{e}_j - \vc{e}_i) \right) \right).\;\;\;\;
\end{eqnarray}
We deduce that:
\begin{proposition} \label{prop:higherdim}
If for some $c$, the matrix $\EEsym{ }{}(\vc{y})$  of Equation \eqref{eq:Edef} is positive definite for all $\vc{y}$, then
 for any function $f$:
\begin{equation} \label{eq:dbecd}
\sum_{\vc{x} \in \Z_+^d} V(\vc{x}) \Gam{2}(f,f)(\vc{x}) \geq c \sum_{\vc{x} \in \Z_+^d} V(\vc{x}) \Gam{1}(f,f)(\vc{x}).
\end{equation}
\end{proposition}
\begin{proof}
First we observe  that (by relabelling)
\begin{eqnarray} \sum_{\vc{x} \in \Z_+^d} V(\vc{x}) \Gam{1}(f,g)(\vc{x}) & = & \sum_{\vc{x} \in \Z_+^d} V(\vc{x}) \sum_{j=1}^d \left( f(\vc{x} + \vc{e}_j) - f(\vc{x}) \right)
\left( g(\vc{x} + \vc{e}_j) - g(\vc{x}) \right) \nonumber \\
& = & \sum_{j=1}^d \sum_{\vc{y} \in \Z_+^d} V(\vc{y} - \vc{e}_j)  \left( f(\vc{y}) - f(\vc{y}- \vc{e}_j) \right)
\left( g(\vc{y}) - g(\vc{y} - \vc{e}_j) \right).\;\; \label{eq:gamma1}
\end{eqnarray}
Using this, we can deduce a $d$-dimensional version of \eqref{eq:done}, namely
\begin{eqnarray} 
\lefteqn{ \sum_{\vc{x} \in \Z_+^d} V(\vc{x}) \Gam{2}(f,g)(\vc{x}) } \nonumber \\
 & = & - \sum_{\vc{x} \in \Z_+^d} V(\vc{x})
 \sum_{j=1}^d \left( L_V f(\vc{x} + \vc{e}_j) - L_V f(\vc{x}) \right)
\left( g(\vc{x} + \vc{e}_j) - g(\vc{x}) \right) \nonumber \\
& = & \sum_{\vc{x} \in \Z_+^d} V(\vc{x}) \sum_{i,j} 
L_{ij} f(\vc{x} + \vc{e}_i) L_{ij} g(\vc{x} + \vc{e}_i) \nonumber \\
& & 
+ \sum_{\vc{x} \in \Z_+^d} V(\vc{x}) \sum_{i=1}^d \sum_{j=1}^d \EEd{i}{j}(\vc{x})
( f( \vc{x}+\vc{e}_j) - f(\vc{x} + \vc{e}_j - \vc{e}_i )) (g( \vc{x} + \vc{e}_j) - g(\vc{x}) )   \nonumber \\
& \geq & \sum_{i=1}^d \sum_{j=1}^d \sum_{\vc{y} 
 \in \Z_+^d} V(\vc{y} - \vc{e}_j)  \EEd{i}{j}(\vc{y} - \vc{e}_j)  ( f( \vc{y}) - f(\vc{y}  - \vc{e}_i )) (g( \vc{y}) - g(\vc{y} - \vc{e}_j) ), \label{eq:doned}
\end{eqnarray}
so taking subtracting $c$ times \eqref{eq:gamma1} from \eqref{eq:doned} and taking $f=g$ we obtain that  
\begin{eqnarray*}
\lefteqn{
\sum_{\vc{x} \in \Z_+^d} V(\vc{x}) \Gam{2}(f,f)(\vc{x}) - c
\sum_{\vc{x} \in \Z_+^d} V(\vc{x}) \Gam{1}(f,f)(\vc{x}) } \\
& \geq & \sum_{\vc{y}  \in \Z_+^d} \left[
\sum_{i=1}^d \sum_{j=1}^d  \EEsym{i}{j}(\vc{y}) (f( \vc{y}) - f(\vc{y} - \vc{e}_j) )
( f( \vc{y}) - f(\vc{y}  - \vc{e}_i )) \right],
\end{eqnarray*}
and the term in square brackets is positive for each $\vc{y}$, by positive-definiteness.
\end{proof}

\begin{remark} \label{rem:independent}
If $V(\vc{x}) = \prod_{k=1}^d V_k(x_k)$ is formed as the product of independent measures 
in each coordinate, then $\EEsym{i}{j}(\vc{y}) \equiv 0$ for $i \neq j$. Further, if each $V_i$ is $c$-log-concave then each entry $\EEsym{i}{i}(\vc{y}) = V_i(y_i-1) \left( \EEval{V_i}(y_i-1) -  c \right) \geq 0$, so
the condition of Proposition \ref{prop:higherdim} is satisfied. This mirrors the tensorization result of 
\cite[Theorem 1.3]{erbar}, which was used  to prove a sharp bound on the Ricci curvature for the hypercube $\{ 0,1 \}^d$.
\end{remark}

Hence, repeating the proof of Theorem \ref{thm:poincare}, we can deduce that the positive definiteness of $\EEsym{ }{ }(\vc{y})$ for
all $\vc{y}$ is enough to imply that a $d$-dimensional Poincar\'{e} inequality holds with constant $ \leq 1/c$.

A more detailed argument shows that many of the arguments used in Section \ref{sec:prooflsi} to prove the new modified
log-Sobolev inequality Theorem \ref{thm:lsi} carry over. That is, we consider functions $f_t(\vc{x})$ evolving as 
$\frac{\partial}{\partial t} f_t(\vc{x}) = L_V f_t(\vc{x})$, for $L_V$ as defined in \eqref{eq:lvd}.
Again, taking $
\Theta(t) := \sum_{\vc{x} \in \Z_+^d}  V(\vc{x}) f_t(\vc{x}) \log f_t(\vc{x})$, we obtain that
\begin{eqnarray}
\Theta'(t) 
& = & - \sum_{\vc{x} \in \Z_+^d}  V(\vc{x}) \sum_{j=1}^d \left( f_t( \vc{x}+ \vc{e}_j) - f_t(\vc{x}) \right) 
\left( \log f_t( \vc{x}+ \vc{e}_j) - \log f_t(\vc{x}) \right).
\end{eqnarray}
Similarly, writing
\begin{eqnarray}
\psi(t) 
& := & \sum_{\vc{x} \in \Z_+^d}  V(\vc{x}) \sum_{j=1}^d \left( f_t( \vc{x}+ \vc{e}_j)  \log \left( \frac{ 
 \log f_t( \vc{x}+ \vc{e}_j)}{ f_t(\vc{x}) } \right) - f_t(\vc{x} + \vc{e}_j) + f_t( \vc{x}) \right), \;\;\;
\end{eqnarray}
an involved analysis using the expressions above shows that
\begin{eqnarray*}
\psi'(t)
& = & \sum_{\vc{x} \in \Z_+^d}  V(\vc{x}) \sum_{i,j=1}^d f(\vc{x} + \vc{e}_j) w \left( U_{ij}(\vc{x}), s_i(\vc{x}) \right) \\
&  & - \sum_{\vc{x} \in \Z_+^d}  V(\vc{x}) \sum_{i,j=1}^d \EEd{i}{j}(\vc{x})
\left( f_t( \vc{x}+ \vc{e}_j) - f_t(\vc{x} + \vc{e}_j - \vc{e}_i) \right) 
\left( \log f_t( \vc{x}+ \vc{e}_j) - \log f_t(\vc{x}) \right).
\end{eqnarray*}  
where as before 
$w(U; s) = -(U/s - 1) \log U + (1-U)(1- 1/s) \geq 0$
and we write
 $U_{ij}(\vc{x}) = f(\vc{x}) f(\vc{x} + \vc{e}_i + \vc{e}_j)/(f(\vc{x} + \vc{e}_i)
f(\vc{x} + \vc{e}_j))$ and $s_i(\vc{x}) = f(\vc{x})/f(\vc{x}+\vc{e}_i)$.
We deduce that 
\begin{equation}
c \Theta'(t) - \psi'(t) =  
\sum_{\vc{y}  \in \Z_+^d} \left[
\sum_{i=1}^d \sum_{j=1}^d  \EEsym{i}{j}(\vc{y}) (f_t( \vc{y}) - f_t(\vc{y} - \vc{e}_j) )
( \log f_t( \vc{y}) - \log f_t(\vc{y}  - \vc{e}_i )) \right]. \label{eq:logsobdiff}
\end{equation}
Unfortunately, positive definiteness of $\EEsym{ }{ }$ is not sufficient to guarantee the  positivity of 
\eqref{eq:logsobdiff} required to deduce the log-Sobolev inequality.
 If (as in Remark \ref{rem:independent}) $V$ is the product of $c$-log-concave
mass functions, then $\EEsym{}{}$ becomes diagonal with positive entries. The positivity of 
\eqref{eq:logsobdiff} follows from the fact that $\log$ is a monotone function,  meaning that
$(f_t( \vc{y}) - f_t(\vc{y} - \vc{e}_j) )$
and $( \log f_t( \vc{y}) - \log f_t(\vc{y}  - \vc{e}_j ))$ have the same sign.

It remains an interesting problem to characterize probability mass functions on $\Z_+^d$ (and indeed for general graph settings)
for which some form of Theorem \ref{thm:lsi} holds.

\section*{Acknowledgments}
The author thanks the University of Bristol for funding to attend the conference `{\em When Dominique Bakry is 60}' 
at Universit\'{e} Paul Sabatier Toulouse in December 2014. He also thanks  
the Institute for Mathematics and Its Applications
for the invitation and funding 
to speak at the workshop `{\em Information Theory and Concentration Phenomena}'  in Minneapolis in April 2015.
Attending talks and having discussions with organisers and
 fellow participants at these meetings  greatly helped  in understanding 
the topics discussed here. Many calculations in this paper were first performed using Mathematica.
The author thanks the anonymous referees of this paper for making numerous extremely helpful suggestions.

\bibliography{../../bibliography/papers}

\end{document}